\documentclass[10pt,leqno]{amsart}
\usepackage{epsfig}
\usepackage{graphicx}
\usepackage{color}
\usepackage[utf8]{inputenc}
\usepackage{amssymb}
\usepackage{amsmath}
\usepackage{amsthm}
\usepackage{MnSymbol}
\usepackage{setspace}

\numberwithin{equation}{section}

\newcommand{\cH}{{\mathcal H}}

\newcommand{\cD}{{\mathcal D}}
\newcommand{\cF}{{\mathcal F}}

\newtheorem{theorem}{Theorem}[section]
\newtheorem{lemma}[theorem]{Lemma}
\newtheorem{corollary}[theorem]{Corollary}

\theoremstyle{definition}

\theoremstyle{remark}

\begin{document}
\title{A class of inequalities for intersection-closed set systems}
\author{Rainer Schrader}
\address{Dept. of Mathematics and Computer Science\\
	University of Cologne\\
	Sibille-Hartmann-Straße  8\\
	50969 Cologne, Germany}
\email{schrader@cs.uni-koeln.de}

\maketitle

\begin{abstract}
Let $N$ be a finite set and $\cF$, an intersection-closed family of subsets. Frankl conjectured that there always exists an element in $N$ which is contained in at most half the number of sets in $\cF$ unless $\cF =\{E\}$. We prove the validity of a class of inequalities which imply Frankl's conjecture.
\end{abstract}

\section{Introduction}
Given a family of subsets of a finite ground set with the property that for every two sets $A,B$ in the family also their intersection $A\cap B$ is a member of the family. The following conjecture
seems to have appeared in the mid-1970s and is attributed to Peter Frankl \cite{Frankl95}: there is a "rare'' element of the ground set which occurs in at most half the number of members of the family. Obviously, this is not true if the family consists only of the ground set but this should be the only exception.\smallskip

While easy to state, the conjecture has resisted numerous proof attempts. In the literature there is a wealth of publications and preprints dealing with reformulations in terms of union-closed set systems, lattices and graphs and proving the conjecture in various special cases. The interested read is referred to an excellent survey by Bruhn and Schaudt \cite{BruhnSchaudt1}. It shows the history of Frankl's conjecture and reviews the reformulations and various results published up to 10 years ago.\smallskip

More recently, using probabilistic approaches, upper bounds $t >0$ have been derived with the property that a rare element never occurs in more than $\frac{1}{2} + t$ sets of the system (cf. Gilmer \cite{Gilmer}, Chase and Lovett \cite{ZCL}, Yu \cite{Yu}, Cambie \cite{Cambie}, Alweiss \cite{Alweiss}, Sawin \cite{sawin}\}).\smallskip

In the following we introduce a class of inequalities for these set systems and prove their validity. This will lead to a proof of Frankl's conjecture.

\section{Notation}
Let $N = \{1,\ldots,n\}$ be a finite ground set. For an integer $i$,  $[i]$ denotes the set $\{1,\ldots, i\}$. We say that a set system $\cF \subseteq 2^N$ over $N$ is  \textit{intersection-closed} if for all $A,B \in \cF$ also $A\cap B \in \cF$. For $1\le i \le n$ let $\cF_i = \{ A\in \cF : i \in A\}$. We assume that $N$ is numbered such that $|\cF_1| \ge |\cF_2| \ge \ldots |\cF_n|$. We also assume that $\cF_1 \ne \cF$ and $\cF_i \ne \cF_j$ for $i\ne j$. Otherwise we may argue by induction on a smaller ground set. \smallskip


\section{Idea of the main Theorem}
Before we formally state the main Theorem, it may be helpful to illustrate its idea. We may think of it as a game against some adversary having an intersection-closed set system in mind. During the game the opponent reveals a growing portion of his set system to us. At each step and with this limited knowledge at hand we have to produce upper bounds on the $|\cF_i|$'s.

\smallskip
Since we assume that $\cF_1 \ne \cF$, the game starts with the information that $\varnothing \in \cF$ and we choose the conservative bound of $2^{n-1}$ for $|\cF_1|$ from a full Boolean algebra. At step $i$ we see all members of the set system with elements in $[i-1]$. Then, since the $|\cF_j|$'s are monotonically non-increasing, a possible response for us is to simply adopt the upper bound generated in the previous step. In some situations, however, we may find a certain set $A$ (a discarding set defined in the following section). Such a set has the property that certain supersets of $A$ cannot be contained in the yet unknown part $\cF$. We will observe that none of these sets has been excluded in an earlier step. This allows us to reduce the upper bound without violating its validity. \smallskip

\section{Preliminary Results}\label{Pre}

Given some $i$ with $1\le i\le n$ we call a set $A$  \textit{discarding at level $i$} if $A \subseteq [i-1], A \in \cF$ and $A \cup i\notin \cF$. The collection of all discarding sets at level $i$ is denoted by $\cD_i$.\smallskip

For a discarding set $A \in \cD_i$ we know that $A\cup i \notin \cF_i$. As mentioned before we can be sure that eventually more sets cannot occur in $\cF_i$. This is based on the fact that the collection of all supersets $A \cup i \cup X$ with $X \subseteq \{i+1,\ldots, n\}, A \cup i \cup X \in \cF$ has the Helly property:

\begin{lemma}\label{lem1}
	Let $\cF$ be intersection-closed, $A \in \cD_i$ for some $i$ with $1 \le i \le n$. Then all $X \subseteq \{i+1,\ldots n\}$ with $A \cup i \cup X \in \cF$ intersect in some $j$ with $j>i$.
\end{lemma}

\begin{proof}
	Let $X, Y \subseteq \{i+1,\ldots, n\}$ be such that $A \cup i \cup X, A \cup i \cup Y \in \cF$. Suppose $X\cap Y = \varnothing$. Then  $(A \cup i \cup X) \cap (A \cup i \cup Y) =  A \cup i \in\cF$ contradicting that $A$ is a discarding set at level $i$.
\end{proof}

Let the above element $j$ be the \textit{root} of the discarding set $A$ and $\cH^A_i$  be the collection of sets the set $A \in \cD_i$ excludes from  $\cF_i$. So 

$$\cH^A_i = \begin{cases}
	A \cup i \cup X ~\text{for all}~ X \subseteq \{i+1,\ldots n\} &\text{if $A$ has no root}\\
	A \cup i \cup X ~\text{for all}~ X \subseteq \{i+1,\ldots n\}, j \notin X &\text{if $A$ has root $j$}
\end{cases}$$


\begin{lemma}\label{lem2}
	Let $\cF$ be intersection-closed , $A \in \cD_i$ for some $i$ with $1 \le i \le n-1$. Then
	$|\cH^A_i| \ge 2^{n-i}$ if $A$ has no root and $|\cH^A_i| \ge 2^{n-(i+1)}$ if $A$ has a root.
\end{lemma}

\begin{proof}
	If no superset of $A\cup i$ is in $\cF$ we may exclude all sets $A \cup i \cup X$ with  $X \subseteq \{i+1,\ldots n\}$, i.e. $|\cH^A_i| = 2^{n-i}$. Otherwise we exclude the same sets except those containing $j$, i.e.  $|\cH^A_i| = 2^{n-i} - 2^{n-(i+1)} = 2^{n-(i+1)}$  and the claim follows.
\end{proof}

The next Lemma is crucial to do the proper accounting in the proof of our main Theorem. It basically states that the elements we exclude from $\cF$ by Lemma \ref{lem2} will never be excluded more than once.

\begin{lemma}\label{lem3}
	Let  $A$  be a discarding set at level $i$ and $B$, a discarding set at level $j$ with  $A \ne B$ and $1 \le j\le i \le n$. Then 
	\begin{align}
		\cH^A_i \cap \cH^B_j &= \varnothing.\label{ineq3}
	\end{align}
\end{lemma}

\begin{proof}
	Suppose the claim is false. Then there exists sets $X \subseteq \{i+1,\ldots,n\}$ and $Y \subseteq \{j+1,\ldots,n\}$ such that $A\cup i \cup X = B \cup j \cup Y$. 
	
	If 	 $i=j$ then, since $A,B \subseteq \{1,\ldots, i-1\}$ and $X,Y \subseteq \{i+1,\ldots n\}$, we have $X=Y$ and $A=B$, a contradiction.
	
	\smallskip
	So let $i>j$. Then $B\cup j \subseteq A$. For $A' = A \setminus (B \cup j)$ we have $B \cup j \cup A' \cup i \cup X = B \cup j \cup Y$. Hence $A' \subseteq Y \subseteq \{j+1,\ldots n\}$. Now, if $B$ has no root then $B \cup j \cup A' \notin \cF$, contradicting $B \cup j \cup A' = A \in \cF$. Otherwise let $k>j$ be the root of $B$. Then, since  $B \cup j \cup A'  = A \in \cF$, $A'$ must contain $k$. So $k \in A' \subseteq Y$ implying $B\cup j \cup Y \notin H^B_j$, a contradiction.
\end{proof}	

\section{The main Theorem}
We now proceed to our main Theorem.

\begin{theorem}\label{thm1}
	Let $N=\{1,\ldots,n\}$ be a finite ground set, $\cF \subseteq 2^N$ an intersection-closed set system with $|\cF_1| \ge |\cF_2| \ge \ldots |\cF_n|$ and $t^0:= 2^{n-1}$. Then for every $1\le i \le n$ with $t^i := t^{i-1} -  \sum\nolimits_{A\in \cD_i} |\cH_i^A|$ the following statements hold:
	\begin{align}
		&& t^i &\ge |\cF_i|\label{ineq4}	\\
		&& 2^{n-i} \cdot |\cF \setminus (\cF_{i} \cup \ldots \cup \cF_n)| &\ge t^{i-1}.\label{ineq5}
	\end{align}
\end{theorem}

\begin{proof}
(\ref{ineq4})~  Obviously, $t^0 = 2^{n-1}$ is an upper bound on $|\cF_1|$. If $\cD_1 \ne \varnothing$ we know from Lemma \ref{lem1} that the $\cH_1$'s cannot appear in  $\cF_1$. Moreover, by Lemma \ref{lem3} the sets in the $\cH_1$'s are mutually disjoint. Hence $t^0 - \sum\nolimits_{A\in \cD_1} |\cH_1^A|$ is still a valid upper bound on $|\cF_1|$.

\smallskip
Suppose the claim holds for $1,\ldots, i$. Since $|\cF_{i-1}| \ge |\cF_i|$ setting $t^i = t^{i-1}$ is an upper bound on $|\cF_i|$. As before, for $\cD_i \ne \varnothing$ the $\cH_i$'s cannot appear in $\cF_i$. We know from Lemma \ref{lem3} that the sets in the $\cH_i$'s are mutually disjoint and are disjoint from the $\cH_j$'s in earlier steps. Hence may reduce the upper bound by $\sum\nolimits_{A\in \cD_i} |\cH_i^A|$.

\smallskip
(\ref{ineq5})~ Recall that we assume  $\cF \ne \cF_1$. So inequality (\ref{ineq5}) trivially holds for $i=1$. 
Suppose the claim holds for $1,\ldots, i$. In the left-hand side of (\ref{ineq5}) we may rewrite  
$|\cF \setminus (\cF_{i+1} \cup \ldots \cup \cF_n)| = \big[|\cF \setminus (\cF_i \cup \ldots \cup \cF_n)| + |\cF_i \setminus (\cF_{i+1} \cup \ldots \cup \cF_n)|\big].$ We distinguish two cases.

\smallskip
If $|\cF_i \setminus (\cF_{i+1} \cup \ldots \cup \cF_n)| \ge |\cF \setminus (\cF_i \cup \ldots \cup \cF_n)|$ then
\begin{align*}
	2^{n-(i+1)} \cdot |\cF \setminus (\cF_{i+1} \cup \ldots \cup \cF_n)| &=
	2^{n-(i+1)} \cdot \big[|\cF \setminus (\cF_i \cup \ldots \cup \cF_n)| + |\cF_i \setminus (\cF_{i+1} \cup \ldots \cup \cF_n|\big]\\
	&\ge 2^{n-(i+1)} \cdot 2\cdot |\cF \setminus (\cF_i \cup \ldots \cup \cF_n)|\\
	& \ge  t^{i-1}\qquad\qquad\qquad\text{(by induction)}\\
	& \ge  t^{i-1} - \sum\nolimits_{A\in \cD_i} |\cH_i^A|\\
	&= t^i,
\end{align*}

If $|\cF_i \setminus (\cF_{i+1} \cup \ldots \cup \cF_n)| < |\cF \setminus (\cF_i \cup \ldots \cup \cF_n)|$ there are sets $A \in \cF \setminus (\cF_i \cup \ldots \cup \cF_n)$ such that $A \cup i  \notin \cF$, i.e. $A \in \cD_i$. Let $\cD' \subseteq \cD$ such that $|\cF_i \setminus (\cF_{i+1} \cup \ldots \cup \cF_n)| + |\cD'| = |\cF \setminus (\cF_i \cup \ldots \cup \cF_n)|$. 

Then 
\begin{align*}
	2^{n-(i+1)} \cdot |\cF \setminus (\cF_{i+1} \cup \ldots \cup \cF_n)| &=
	2^{n-(i+1)} \cdot \big[|\cF \setminus (\cF_i \cup \ldots \cup \cF_n)| + |\cF_i \setminus (\cF_{i+1} \cup \ldots \cup \cF_n|\big]\\
	&\ge 2^{n-(i+1)} \cdot \big[2\cdot |\cF \setminus (\cF_i \cup \ldots \cup \cF_n)| - |\cD'_i|\big]\\
	& \ge  t^{i-1} - 2^{n-(i+1)} \cdot |\cD'_i|\qquad\text{(by induction)}\\
	& \ge  t^{i-1} - 2^{n-(i+1)} \cdot |\cD_i|\qquad(\text{since}~ \cD'_i \subseteq \cD_i)\\
	& \ge  t^{i-1} - \sum\nolimits_{A\in \cD_i} |\cH_i^A|\qquad(\text{by Lemma \ref{lem2}})\\
	&=t^i,
\end{align*}
which finishes the proof.

\end{proof}

\section{Frankl's Conjecture}
Theorem \ref{thm1} leads to a slightly stronger version of Frankl's conjecture:

\begin{corollary}[\textbf{Frankl's conjecture}]
	Let $\cF$ be an intersection-closed set system over a finite ground set $N$ with $\cF \ne \{N\}$ and $|\cF_1| \ge |\cF_2| \ge \ldots |\cF_n|$. Then  $$|\cF| \ge |\cF_{n-1}| + |\cF_n|.$$
\end{corollary}	
\begin{proof}
	By Theorem \ref{thm1} we have $|\cF \setminus \cF_n| \ge t^{n-1} \ge  |\cF_{n-1}|$ and the claim follows.
\end{proof}	

A stronger version of Frankl's conjecture states that the rare element occurs in exactly half the number of sets of $\cF$ if and only if $\cF$ is isomorphic to a Boolean algebra. Before we can show this we need the following Lemma.

\begin{lemma}\label{lem5}
	Let $\cF$ be an intersection-closed set system over a finite ground set $N$ with $\cF \ne \{N\}$ and  $|\cF_1| \ge |\cF_2| \ge \ldots |\cF_n|$. Then $|\cF \setminus (\cF_i \cup \ldots \cup \cF_n)| - |\cD_i|= |\cF_i \setminus (\cF_{i+1} \cup \ldots \cup \cF_n)|$ implies
	\begin{equation}
		A \in \cF \Leftrightarrow   A \cup i \in \cF ~\text{for all}~ A \subseteq [i-1]\label{equ2}
	\end{equation}
\end{lemma}
\begin{proof}
	If $|\cF \setminus (\cF_i \cup \ldots \cup \cF_n)| - |\cD_i|= |\cF_i \setminus (\cF_{i+1} \cup \ldots \cup \cF_n)|$ holds then 
	\begin{align*}
		|\{A  : A \subseteq [i-1], A\cup i \in \cF\}| &= |\cF_i \setminus (\cF_{i+1} \cup \ldots \cup \cF_n)|\\
		&=|\cF \setminus (\cF_i \cup \ldots \cup \cF_n)| - |\cD_i|\\
		&= |\{A \in \cF : A \subseteq [i-1]\}|\\
		&\quad - |\{A \in \cF : A \subseteq [i-1], A\cup i \notin \cF\}|\\
		&=|\{A \in \cF : A \subseteq [i-1], A\cup i \in \cF\}|.
	\end{align*}
	Hence (\ref{equ2}) holds. 
\end{proof}

We are now prepared for
\begin{corollary}
	Let $\cF$ be an intersection-closed set system over a finite ground set $N$ with $\cF \ne \{N\}$ and .  $|\cF_1| \ge |\cF_2| \ge \ldots |\cF_n|$. Then $2\cdot |\cF_n| = |\cF|$ if and only if $\cF$ is a Boolean algebra.
\end{corollary}
\begin{proof}
	Clearly, if $\cF$ is a Boolean algebra then the claim follows. Conversely, assume $2\cdot |\cF_n| = |\cF|$. We perform backward induction for $k=n,\ldots,1$ and prove (\ref{equ2}) and
	\begin{align}
		2^{n-k} \cdot |\cF {\setminus} (\cF_k \cup \ldots \cup \cF_n)| = t^{k-1}.\label{equ3}
	\end{align}	
	
	For $k = n$ Theorem \ref{thm1} implies $ |\cF_n| = |\cF \setminus \cF_n|  \ge t^{n-1} \ge t^n \ge |\cF_n|$. in particular, $\cD_n= \varnothing$ and hence (\ref{equ2}) and (\ref{equ3}) hold.

	\smallskip
	Assume that they have been shown for $n,\ldots,n-k+1$. To see that they also hold for $n-k$ recall the two cases in the proof of Theorem \ref{thm1}. 
	If $|\cF_{n-k} {\setminus} (\cF_{n-k+1} \cup \ldots \cup \cF_n)| \ge |\cF {\setminus} (\cF_{n-k+1} \cup \ldots \cup \cF_n)|$ then
	\begin{align*}
		2^{k-1} |\cF {\setminus} (\cF_{n-k+1} \cup \ldots \cup \cF_n)| &= 2^{k-1} \big[|\cF {\setminus} (\cF_{n-k} \cup \ldots \cup \cF_n)| + |\cF_{n-k} {\setminus} (\cF_{n-k+1} \cup \ldots \cup \cF_n)|\big]\\
		&\ge 2^{k-1} \cdot 2\cdot |\cF {\setminus} (\cF_{n-k} \cup \ldots \cup \cF_n)|\\
		& \ge  t^{n-(k+1)}\\
		& \ge  t^{n-(k+1)} - \sum\nolimits_{A\in \cD_{n-k}} |\cH_{n-k}^A|\\
		&=t^{n-k}.
	\end{align*}
	
	Otherwise
	\begin{align*}
		2^{k-1} |\cF {\setminus} (\cF_{n-k+1} \cup \ldots \cup \cF_n)| &= 2^{k-1} \big[|\cF {\setminus} (\cF_{n-k} \cup \ldots \cup \cF_n)| + |\cF_{n-k} {\setminus} (\cF_{n-k+1} \cup \ldots \cup \cF_n)|\big]\\
		&\ge 2^{k-1} \big[2\cdot |\cF {\setminus} (\cF_{n-k} \cup \ldots \cup \cF_n)| - |\cD'_{n-k}|\big]\\
		& \ge  t^{n-(k+1)} - 2^{k-1} |\cD'_{n-k}|\\
		& \ge  t^{n-(k+1)} - 2^{k-1} |\cD_{n-k}|\\
		& \ge  t^{n-(k+1)} - \sum\nolimits_{A\in \cD_{n-k}} |\cH_{n-k}^A|\\
		&=t^{n-k}.
	\end{align*}
	
	Since by induction $2^{k-1} |\cF {\setminus} (\cF_{n-k+1} \cup \ldots \cup \cF_n)| =t^{n-k}$, we have in both cases equality throughout. In particular, (\ref{equ3}) holds in both cases. 
	
	\smallskip
	Moreover, in the first case $|\cF_{n-k} {\setminus} (\cF_{n-k+1} \cup \ldots \cup \cF_n)| = |\cF {\setminus} (\cF_{n-k+1} \cup \ldots \cup \cF_n)|$ and $\sum\nolimits_{A\in \cD_{n-k}} |\cH_{n-k}^A| = 0$ hold So $\cD_{n-k} = \varnothing$. In the second case we get $\cD'_{n-k} = \cD_{n-k}$ and $|\cF \setminus (\cF_i \cup \ldots \cup \cF_n)| - |\cD_i|= |\cF_i \setminus (\cF_{i+1} \cup \ldots \cup \cF_n)|$ holds. Hence in both cases (\ref{equ2}) follows from Lemma \ref{lem5}.\smallskip
	
	The above implies that $\cF$ is built by step wise adding the next element to all sets generated in the previous steps.  Hence $\cF$ is a Boolean algebra which finishes the proof.
\end{proof}

\bibliographystyle{abbrv}
\bibliography{literatur}
\end{document}